\documentclass[12pt]{amsart}

\usepackage{graphicx}
\usepackage{amssymb}
\usepackage{float}
\usepackage{fullpage}
\usepackage{physics}

\newtheorem{theorem}{Theorem}
\newtheorem{lemma}[theorem]{Lemma}
\newtheorem{corollary}[theorem]{Corollary}

\theoremstyle{definition}

\theoremstyle{remark}

\numberwithin{equation}{section}

\newcommand{\Z}{{\mathbb Z}}
\newcommand{\N}{{\mathbb N}}

\begin{document}

\title[Translated sums of primitive sets]{Translated sums of primitive sets}

\author[J. D. Lichtman]{Jared Duker Lichtman}
\address{Mathematical Institute, University of Oxford, Oxford, OX2 6GG, UK}

\email{jared.d.lichtman@gmail.com}

\date{May 7, 2023.}


\begin{abstract}
The Erd\H{o}s primitive set conjecture states that the sum $f(A) = \sum_{a\in A}\frac{1}{a\log a}$, ranging over any primitive set $A$ of positive integers, is maximized by the set of prime numbers. Recently Laib, Derbal, and Mechik proved that the translated Erd\H{o}s conjecture for the sum $f(A,h) = \sum_{a\in A}\frac{1}{a(\log a+h)}$ is false starting at $h=81$, by comparison with semiprimes. In this note we prove that such falsehood occurs already at $h= 1.04\cdots$, and show this translate is best possible for semiprimes. We also obtain results for translated sums of $k$-almost primes with larger $k$.
\end{abstract}



\maketitle

\section{Introduction}

A set $A\subset\Z_{>1}$ of positive integers is called {\it primitive} if no member divides another (we trivially exclude the singleton $\{1\}$). An important family of examples is the set $\N_k$ of {\it $k$-almost primes}, that is, numbers with exactly $k$ prime factors counted with multiplicity. For example, $k=1,2$ correspond to the sets of primes and semiprimes, respectively. 

In 1935 Erd\H{o}s \cite{Erdos35} proved that $f(A) = \sum_{a\in A}\frac{1}{a\log a}$ converges uniformly for any primitive $A$. In 1986 he further conjectured \cite[Conjecture 2.1]{ErdosLimoges} that the maximum of $f(A)$ is attained by the primes $A=\N_1$. One may directly compute $f(\N_1)=1.636\cdots$, whereas the best known bound is $f(A)<e^\gamma=1.781\cdots$ for any primitive $A$ \cite{LPprim}.  As a special case, Zhang \cite{zhang2} proved that the primes maximize $f(\N_k)$, that is, $f(\N_1) > f(\N_k)$ for all $k>1$. 

One may pose a translated analogue of the Erd\H{o}s conjecture, namely, the maximum of $f(A,h) = \sum_{a\in A}\frac{1}{a(\log a+h)}$ is attained by the primes $A=\N_1$. Recently Laib, Derbal, and Mechik \cite{LDM} proved that this translated conjecture is false, by showing $f(\N_1,h)<f(\N_2,h)$ for all $h\ge81$. Their proof method is direct, by studying partial sum truncations of $f(A,h)$. (Laib \cite{Laib} very recently announced a bound $h\ge60$, as a refinement of the same method.)

By different methods, we extend the range of such falsehood down to $h> 1.04\cdots$, and show this translate is best possible for semiprimes.

\begin{theorem}\label{thm:h2}
Let $P(s)=\sum_p p^{-s}$ denote the prime zeta function. We have $f(\N_1,h)<f(\N_2,h)$ if and only if $h > h_2$, where $t=h_2=1.04\cdots$ is the unique real solution to
\begin{align*}
\int_1^\infty \Big[P(s) - \tfrac{1}{2}\big(P(s)^2+P(2s)\big)\Big]e^{(1-s)t}\dd{s} \ = \ 0.
\end{align*}
Moreover $f(\N_1,h)<f(\N_k,h)$ for all $k$ sufficiently large, if and only if $h>h_\infty$, where $t=h_\infty=0.803\cdots$ is the unique real solution to
\begin{align*}
\int_1^\infty P(s)e^{(1-s)t}\dd{s} \ = \ 1.
\end{align*}

\end{theorem}

This suggests that the Erd\H{o}s conjecture, if true, is only barely so. Moreover, for the same value $h=h_2=1.04\cdots$ we show the primes {\it minimize} $f(\N_k,h)$, which may be viewed as the inverse analogue of Zhang's maximization result. 

\begin{theorem}\label{thm:min}
For $h_2=1.04\cdots$, we have $f(\N_1,h_2)<f(\N_k,h_2)$ for all $k>1$.
\end{theorem}

\section{Proof of Theorem 1}

We introduce the zeta function for $k$-almost primes $P_k(s) = \sum_{n\in\N_k}n^{-s}$, whose relevance to us arises from the identity,
\begin{align}\label{eq:fNkhint}
f(\N_k,h) &= \sum_{n\in \N_k}\frac{1}{n(\log n + h)} = \sum_{n\in \N_k}\frac{1}{n\log(ne^h)} \nonumber\\
&= \sum_{n\in \N_k}e^h\int_1^\infty (ne^h)^{-s}\dd{s} = \int_1^\infty P_k(s)e^{(1-s)h}\dd{s}.
\end{align}
Here the interchange of sum and integral holds by Tonelli's theorem, since $f(\N_k,h)\le f(\N_k)$ converges uniformly after Erd\H{o}s. The significance of the identity \eqref{eq:fNkhint} was first observed when $k=1$, $h=0$ by H. Cohen \cite[p.6]{Cohen}, who rapidly computed $f(\N_1)=1.636616\cdots$ to 50 digits accuracy. By comparison, the direct approach by partial sums $\sum_{p\le x} 1/(p\log p)$ converge too slowly, i.e. $O(1/\log x)$. Similarly for $k>1$, we shall see \eqref{eq:fNkhint} leads to sharper results.

Note one has $P_1(s) = P(s)$ and $P_2(s) = \frac{1}{2}P(s)^2+\frac{1}{2}P(2s)$, as well as
\begin{align*}
P_3(s) &= \frac{1}{6}P(s)^3 + \frac{1}{2}P(s)P(2s) + \frac{1}{3}P(3s),\\
P_4(s) &= \frac{1}{24}P(s)^4 + \frac{1}{4}P(s)^2P(2s) + \frac{1}{8}P(2s)^2+\frac{1}{3}P(s)P(3s) + \frac{1}{4}P(4s).
\end{align*}
In general for $k\ge1$, Proposition 3.1 in \cite{Lalmost} gives an explicit formula for $P_k$ in terms of $P$,
\begin{align}\label{eq:explicit}
P_k(s) = \sum_{n_1+2n_2+\cdots=k}\prod_{j\ge1}\frac{1}{n_j!}\big(P(js)/j\big)^{n_j}.
\end{align}
Here the above sum ranges over all partitions of $k$. Also see Proposition 2.1 in \cite{LMert}. 

In practice, we may rapidly compute $P_k$ (and $P'_k$) using recursion relations.

\begin{lemma}
For $k\ge1$ let $P_k(s) = \sum_{\Omega(n)=k}n^{-s}$ and $P_1(s) = P(s) = \sum_p p^{-s}$. We have
\begin{align}\label{eq:recur}
P_k(s) = \frac{1}{k}\sum_{j=1}^{k}P_{k-j}(s) P(js) \quad\text{and}\quad
P'_k(s) = \sum_{j=1}^{k}P_{k-j}(s) P'(js).
\end{align}
\end{lemma}
\begin{proof}
The recursion \eqref{eq:recur} for $P_k$ is given in Proposition 3.1 of \cite{Lalmost}, and is equivalent to \eqref{eq:explicit}.

The recursion \eqref{eq:recur} for $P_k'$ is obtained by differentiating \eqref{eq:explicit},
\begin{align*}
P'_k(s) & = \sum_{n_1+2n_2+\cdots=k}\sum_{i\le k}P'(is) \frac{(P(is)/i)^{n_i-1}}{(n_i-1)!}\prod_{j\neq i}\frac{1}{n_j!}\big(P(js)/j\big)^{n_j}\\
& = \sum_{i\le k}P'(is)\sum_{n_1+2n_2+\cdots=k-i}\prod_{j\ge1}\frac{1}{n_j!}\big(P(js)/j\big)^{n_j} = \sum_{i\le k}P'(is)P_{k-i}(s).
\end{align*}
\end{proof}

As observed empirically in \cite{BM}, the Dirichlet series $P_2(s) - P(s) = \frac{1}{2}[P(s)^2+P(2s)]-P(s)$ has a unique root at $s=\sigma = 1.1403\cdots$, through which it passes from positive to negative. We prove this more generally for $k\le 20$.
\begin{lemma}\label{lem:rootsigmak}
For $2\le k\le 20$, the Dirichlet series $P_k(s) - P(s)$ has a unique root at $s=\sigma_k>1$, through which it passes from positive to negative.
\end{lemma}
\begin{proof}
For each $k\ge1$, $P_k(s)$ is monotonically decreasing in $s>1$. As such there is a unique $s_k>1$ for which $P(s)$ passes through $(k!)^{1/(k-1)}$ from above. Using the main term in \eqref{eq:explicit}, $P_k(s) > P(s)^k/k!$ which is larger than $P(s)$ iff $P(s)^{k-1} > k!$ iff $s < s_k$ by definition. That is,
\begin{align}\label{eq:sk}
P_k(s) &> P(s)>0 \qquad \text{for} \quad s \in (1,s_k).
\end{align}
Also there is a unique $s'_k>1$ for which $P_{k-1}(s)$ passes through $1$ from above. Using the first term in the recursion \eqref{eq:recur}, $-P_k'(s) > -P'(s)P_{k-1}(s)$ which is larger than $-P'(s)>0$ iff $P_{k-1}(s)>1$ iff $s < s'_k$ by definition. That is,
\begin{align}
P'_k(s) &< P'(s)<0 \qquad \text{for} \quad s \in (1,s'_k).
\end{align}

For $c<2^k$, $P_k(s)c^s$ is monotonically decreasing in $s>1$, and so is $P_k(s)/(2^{-s}+3^{-s}) = \{1/[P_k(s)2^s]+1/[P_k(s)3^s]\}^{-1}$. Thus there is a unique $ t_k >1$ for which $P_k(s)/(2^{-s}+3^{-s})$ passes through 1 from above. Now by definition $P(s)/(2^{-s}+3^{-s}) > 1 = P_k(t_k)/(2^{-t_k}+3^{-t_k})$, which is larger than $P_k(s)/(2^{-s}+3^{-s})$ iff $s > t_k$ by monotonicity. That is,
\begin{align}\label{eq:tk}
0<P_k(s) &< P(s) \qquad \text{for} \quad s \in (t_k, \infty).
\end{align}

In summary $P_k-P$ has at most one root in $(1,s'_k)$, and no roots in $(1,s_k)\cup(t_k,\infty)$. We directly compute that $s_k<t_k<s'_k$ for $2\le k\le 20$, and so $P_k(s)-P(s)$ has a unique root $\sigma_k\in (s_k,t_k)$ as claimed.
\end{proof}

We deduce the following corollary, which gives (the first part of) Theorem \ref{thm:h2} when $k=2$.

\begin{corollary}\label{cor:fNkh}
For $2\le k\le 20$, $f(\N_k,h)-f(\N_1,h)$ has a unique root at $h_k>0$, through which it passes from negative to positive.
\end{corollary}
\begin{proof}
For $h\ge0$, recall $f(\N_k,h)=\int_1^\infty P_k(s)e^{(1-s)h}\dd{s}$ by \eqref{eq:fNkhint}. Now by Lemma \ref{lem:rootsigmak}, $P_k(s)-P(s)$ passes from positive to negative at $s=\sigma_k>1$. Thus
\begin{align*}
[f(\N_k,h)-f(\N_1,h)]e^{(\sigma_k-1)h} &= \int_1^\infty  [P_k(s)-P(s)]e^{(\sigma_k-s)h}\dd{s}\\
&= \int_1^{\sigma_k}  [P_k(s)-P(s)]e^{(\sigma_k-s)h}\dd{s} - \int_{\sigma_k}^\infty [P(s)-P_k(s)]e^{(\sigma_k-s)h}\dd{s}
\end{align*}
is difference of two integrals with positive integrands, which are mononotically increasing and decreasing in $h\ge 0$, respectively. Hence the difference is mononotically increasing in $h\ge0$. And $f(\N_k,0)-f(\N_1,0)<0$ by Zhang \cite{zhang2}, so the result follows.
\end{proof}

For the second part of Theorem \ref{thm:h2}, note for any fixed $h\ge0$ we have $\log n + h \sim \log n$ for $n\in\N_k$ as $k\to\infty$, since $n\ge 2^k$. Thus $f(\N_k,h) \sim f(\N_k,0)$ as $k\to\infty$. Hence by Theorem 2.2 in \cite{Lalmost},
\begin{align*}
\lim_{k\to\infty}f(\N_k,h)=\lim_{k\to\infty}f(\N_k) = 1.
\end{align*}
Note $1-f(\N_1,h)$ passes from negative to positive at a unique root $h_\infty = 0.803524\cdots$. So for each $h>h_\infty$, we have $f(\N_k,h) - f(\N_1,h) > 0$ for $k$ sufficiently large (and similarly for the converse). This completes the proof of Theorem \ref{thm:h2}.

\subsection{Computations}

For $k\le20$, we compute the unique roots $\sigma_k$ and $h_k$ of $P_k(s)-P(s)$ and $f(\N_k,h)-f(\N_1,h)$, respectively, as well as verify that the auxiliary parameters (as defined in Lemma \ref{lem:rootsigmak}) satisfy $s_k<\sigma_k<t_k<s'_k$. We similarly compute the root of $1-f(\N_1,h)$ as $h_\infty = 0.803524\cdots$.

In our computations, we express $P_k$ in terms of $P$ using the recursion in \eqref{eq:recur}. In turn by M\"obius inversion $P(s) = \sum_{m\ge1} (\mu(m)/m)\log \zeta(ms)$, so $P$ is obtained via well-known rapid computation of $\zeta$. Finally, we compute $f(\N_k,h)$ from its integral form \eqref{eq:fNkhint}. The data are displayed in the table below, obtained using Mathematica (for technical convenience, we first compute $y_k = \log h_k$)\footnote{
{\tt P[k\_Integer,s\_]:= If[k==1,PrimeZetaP[s], Expand[(Sum[P[1,j*s]*P[k-j,s],{j,1,k-1}]+P[1,k*s])/k]]

FindRoot[P[1, s] == (k!)\^{}(1/(k - 1)), {s, 1 + 1/k\^{}3}]

FindRoot[P[k, t]/(2\^{}(-t) + 3\^{}(-t)) == 1, {t, 1 + 1/k\^{}3}]
  
FindRoot[P[k - 1, s1] == 1, {s1, 1 + 1/k\^{}3}]

FindRoot[P[k, sigma] == P[1, sigma], {sigma, 1 + 1/k\^{}3}]]

FindRoot[NIntegrate[(P[1, s] - P[k, s])/y\^{}s,{s,1,Infinity}, WorkingPrecision->30, AccuracyGoal->13, PrecisionGoal->13], {y, 1 + 1/k\^{}3}] }}

\[\label{fig:sigk}
\begin{tabular}{c|lll|ll}
$k$ & $s_k$ & $t_k$ & $s'_k$ & $\sigma_k$ & $h_k$\\
\hline
2  & 1.11313  & 1.40678  & 1.39943 & 1.14037  & 1.04466\\
3  & 1.06861  & 1.23367  & 1.25922  & 1.09224 & 0.98213\\
4  & 1.04306  & 1.15231  & 1.17696  & 1.06206 & 0.93018\\
5  & 1.02761  & 1.104    & 1.12386  & 1.04231 & 0.89038\\
6  & 1.01795  & 1.07259  & 1.08784  & 1.02907 & 0.86146\\
7  & 1.01179  & 1.05125  & 1.06272  & 1.02007 & 0.84126\\
8  & 1.00779  & 1.0364   & 1.04493  & 1.0139  & 0.8276\\
9  & 1.00518  & 1.02594  & 1.03223  & 1.00964 & 0.8186\\
10 & 1.00346  & 1.0185   & 1.02312  & 1.0067  & 0.8128\\
11 & 1.00231  & 1.0132   & 1.01658  & 1.00466 & 0.80915\\
12 & 1.00155  & 1.00942  & 1.01187  & 1.00325 & 0.80689\\
13 & 1.00105  & 1.00672  & 1.00849  & 1.00226 & 0.80551\\
14 & 1.0007   & 1.00479  & 1.00607  & 1.00158 & 0.8047\\
15 & 1.00048  & 1.00341  & 1.00433  & 1.0011  & 0.8042\\
16 & 1.00032  & 1.00243  & 1.00309  & 1.00077 & 0.80391\\
17 & 1.00022  & 1.00173  & 1.0022   & 1.00053 & 0.80374\\
18 & 1.00015  & 1.00123  & 1.00157  & 1.00037 & 0.80365\\
19 & 1.0001   & 1.00087  & 1.00112  & 1.00026 & 0.80359\\
20 & 1.00007  & 1.00062  & 1.00079  & 1.00018 & 0.80356
\end{tabular}\]

We believe a unique of root $h_k$ exists as in Corollary \ref{cor:fNkh} for all $k>20$ as well. This would enable a strengthening of Theorem \ref{thm:min} to $f(\N_k,h)>f(\N_1,h)$ for all values of $k>1$, $h\ge h_2$ (so far we only establish this for \{$1<k\le 20$, $h\ge h_2$\} or \{$k>1$, $h=h_2$\}). Uniqueness of $h_k$ would follow if $\sigma_k$ is unique, as in Lemma \ref{lem:rootsigmak}, for all $k$. In turn it would suffice to show $t_k<s_k'$ for all $k$ (note $s_k<t_k$ holds automatically by \eqref{eq:sk},\eqref{eq:tk}), though it is not clear how to establish such an inequality in general.

Moreover, it appears both $h_k$ and $\sigma_k$ are monotonically decreasing in $k$. This may be related to some empirical trends for $f(\N_k)$, found in a recent disproof of a conjecture of Banks--Martin, see \cite{BM}, \cite{Lalmost}.

\section{Proof of Theorem \ref{thm:min}}

We have already verified the claim directly for $k\le 20$, since in this case $h_k\le h_2=1.04\cdots$.

For $k>20$, the proof strategy is similar to that of Theorem 5.5 in \cite{Lalmost}. That is, the integral $f(\N_k,h) = \int_1^\infty P_k(s)e^{(1-s)h} \dd{s}$ has its mass concentrated near 1 as $k\to\infty$, so it suffices to truncate the integration to $[1,1.01]$ say, as a lower bound. Thus by \eqref{eq:fNkhint},
\begin{align}\label{eq:Pkcombshort}
f(\N_k,h_2) & = \int_1^\infty P_k(s)e^{(1-s)h_2}\dd{s} 
> e^{-.01h_2}\int_1^{1.01} P_k(s)\dd{s}.
\end{align}

Next, we may lower bound $P_k(s)$ by $P(s)^k/k!$, which constitutes the first of the terms in the identity \eqref{eq:explicit}, one per partition of $k$. Note the terms of partitions built from small parts contribute the most mass. So by also including the terms for the partitions $k = 1\cdot(k-j) + j$ and $k = 1\cdot(k-j-2) + 2 + j$ for $j\le6$, we shall obtain a sufficiently tight lower bound to deduce the result. Indeed, we have
\begin{align}\label{eq:Pkcombshort}
\int_1^{1.01} P_k(s)\dd{s} \ & > \ \frac{1}{k!}\int_1^{1.01} P(s)^k\dd{s} + \sum_{j=2}^6\frac{\int_1^{1.01} P(s)^{k-j}P(js)\dd{s}}{j(k-j)!} \nonumber\\
& \quad + \ \frac{\int_1^{1.01} P(s)^{k-4}P(2s)^2\dd{s}}{2!2^2(k-4)!} + \sum_{j=3}^6\frac{\int_1^{1.01} P(s)^{k-j-2}P(2s)P(js)\dd{s}}{2j(k-j-2)!}.
\end{align}
From (5.10) in \cite{Lalmost}, we have
\begin{align}
0<P(s) - \log(\tfrac{\alpha}{s-1})<1.4(s-1), \qquad \text{for }s\in[1,2],
\end{align}
where $\alpha = \exp(-\sum_{m\ge2}P(m)/m) = .7292\cdots$.
Thus for every $k\ge1$, since $\log(\alpha/.01)>4$,
\begin{align}\label{eq:Psk}
\int_1^{1.01} P(s)^k\dd{s} > \int_0^{.01} \log\big(\tfrac{\alpha}{s}\big)^k\dd{s} = \alpha\int_{\log(\alpha/.01)}^\infty u^k e^{-u}\;du > \alpha\,\Gamma(k+1,4) > .729\,k!,
\end{align}
where $\Gamma(k+1,4)$ the incomplete Gamma function, and noting $\Gamma(k+1,4)/k!$ is monotonically increasing in $k$. 
Also note
\begin{align*}
\int_0^1 s\log\big(\tfrac{\alpha}{s}\big)^k\dd{s} = \alpha^2\int_0^\infty u^k e^{-2u}\;du = \frac{\alpha^2}{2^{k+1}}k!.
\end{align*}
Using the first order Taylor approximation $P(js) > P(j) + P'(j)(s-1)$ for $j\ge2$,
\begin{align*}
\int_1^{1.01} P(s)^{k-j}P(js)\dd{s} \ & > \ P(j)\int_0^{.01} \log\big(\tfrac{\alpha}{s}\big)^{k-j}\dd{s}+P'(j)\int_0^1 s\log\big(\tfrac{\alpha}{s}\big)^{k-j}\dd{s}\\
& > \ .729(k-j)!\Big(P(j) + \frac{\alpha P'(j)}{2^{k-j}}\Big)
\end{align*}
by \eqref{eq:Psk}. Similarly,
\begin{align*}
\int_1^{1.01} & P(s)^{k-j-2} P(2s)P(js) \dd{s}\\
> & \ P(2)P(j)\int_0^{.01} \log\big(\tfrac{\alpha}{s}\big)^{k-j-2}\dd{s} \ + \ [P'(2)P(j)+P(2)P'(j)] \int_0^1 s\log\big(\tfrac{\alpha}{s}\big)^{k-j-2}\dd{s}\\
> & \ .729(k-j-2)!\Big(P(2)P(j) + \frac{\alpha}{2^{k-j-1}}[P'(2)P(j)+P(2)P'(j)]\Big).
\end{align*}
Hence plugging back into \eqref{eq:Pkcombshort},
\begin{align}
f(\N_k,h_2)e^{.01h_2} > \int_1^{1.01} P(s)^k\dd{s} \ > \ \ell_k,
\end{align}
for the explicit lower bound
\begin{align*}
\ell_k \ := \ .729\bigg[1 + \sum_{j=2}^6 &\frac{P(j)+\alpha P'(j)/2^{k-j}}{j} \ + \ \frac{1}{8}\Big(P(2)^2+ \frac{\alpha P(2)P'(2)}{2^{k-4}}\Big) \nonumber\\
\ + \
\sum_{j=3}^6 & \frac{1}{2j}\bigg(P(2)P(j) +  \frac{\alpha}{2^{k-j-1}}[P'(2)P(j)+P(2)P'(j)]\bigg)\bigg].
\end{align*}
Note $\ell_k$ is clearly increasing in $k$ (recall $P'(s)<0$). Hence for $k>20$ we have
\begin{align}
f(\N_k,h_2) &> e^{-.01h_2}\ell_k > e^{-.01h_2}\ell_{20}
> .98>0.91 >f(\N_1,h_2).
\end{align}
Here we compute $\ell_{20} = 0.99069\cdots$ and $f(\N_1,h_2)=0.908599\cdots$. This completes the proof.

\section*{Acknowledgments}
The author thanks Paul Kinlaw and Carl Pomerance for helpful discussions, and the anonymous referee for comments. The author also thanks Michel Balazard and François Morain for the reference \cite{ErdosLimoges} for the Erd\H{o}s primitive set conjecture, as well as Ilias Laib for the correct value of $h_\infty\approx 0.803$.

The author is supported by a Clarendon Scholarship at the University of Oxford.

\bibliographystyle{amsplain}

\end{document}